\documentclass[10pt,a4paper]{article}
\usepackage{makeidx} 
\usepackage{amsmath}
\usepackage{amssymb}
\usepackage{amsthm}
\usepackage{amsfonts}
\usepackage{enumerate}
\usepackage{llncsdoc}
\usepackage{verbatim}

\newcommand{\tluste}[1]{\mbox{\mathversion{bold}$ #1 $}}

\newcommand{\mace}[1]{{{#1}}}

\newcommand{\omace}[1]{\mbox{$\overline{\mace{#1}}$}} 
\newcommand{\umace}[1]{\mbox{$\underline{\mace{#1}}$}} 
\newcommand{\imace}[1]{\mbox{$\tluste{#1}$}} 
 
\newcommand{\smace}[1]{\tluste{#1}{}^S} 
\def\Mid#1{#1^c}
\def\Rad#1{#1^\Delta}
\newcommand{\onum}[1]{\mbox{$\overline{{#1}}$}} 
\newcommand{\unum}[1]{\mbox{$\underline{{#1}}$}}

\newcommand{\ivr}[1]{\mbox{$\tluste{#1}$}} 
 
\newcommand{\inum}[1]{\mbox{$\tluste{#1}$}} 
\newcommand{\R}[0]{{\mathbb{R}}}

\newcommand{\IR}[0]{{\mathbb{IR}}}

\newcommand{\mmid}[0]{;\,}		%pouzivejte v definici mnozin!

\newcommand{\sseznam}[2]{{\{{#1}, \ldots, {#2}\}}}
\def\clqq{``}
\def\crqq{''}
\def\quo#1{\clqq{}#1\crqq{}}  % snadny zapis ang. uvozovek

\DeclareMathOperator{\sgn}{sgn}	%sign
\DeclareMathOperator{\diag}{diag}%diag

\def\nref#1{$(\ref{#1})$}

\newtheorem{theorem}{Theorem}

\theoremstyle{definition}

\newtheorem{definition}{Definition} 
\newtheorem{example}{Example}
\begin{document}

\title{Positive Semidefiniteness and Positive Definiteness of a Linear Parametric Interval Matrix}
%\titlerunning{Positive (Semi-)Definiteness of a Linear Parametric Matrix}

\author{
  Milan Hlad\'{i}k\footnote{
Charles University, Faculty  of  Mathematics  and  Physics,
Department of Applied Mathematics, 
Malostransk\'e n\'am.~25, 11800, Prague, Czech Republic, 
e-mail: \texttt{milan.hladik@matfyz.cz}
}
}
\maketitle

\begin{abstract}%\normalsize
We consider a symmetric matrix, the entries of which depend linearly on some parameters. The domains of the parameters are compact real intervals. We investigate the problem of checking whether for each (or some) setting of the parameters, the matrix is positive definite (or positive semidefinite). We state a characterization in the form of equivalent conditions, and also propose some computationally cheap sufficient\,/\,necessary conditions. Our results extend the classical results on positive (semi-)definiteness of interval matrices. They may be useful for checking convexity or non-convexity in global optimization methods based on branch and bound framework and using interval techniques.
\end{abstract}


%%%%%%%%%%%%%%%%%%%%%%%%%%%%%%%%%%%%%%%%%%%%%%%%%%%%%%%%%%%%%%% 
% INTRODUCTION
%%%%%%%%%%%%%%%%%%%%%%%%%%%%%%%%%%%%%%%%%%%%%%%%%%%%%%%%%%%%%%% 
\section{Introduction}

A commonly used deterministic approach to global optimization  \cite{Flo2000,FloPar2009,HanWal2004,HenTot2010,Kea1996,Neu2004} is based on exhaustive splitting of the search space into smaller parts (usually boxes) and applying various interval techniques to remove boxes that provably do not contain any global minimizer, to compute rigorous lower and upper bounds on the optimal value, and to prove optimality of some point within a box, among others.

An important step in this approach is convexity testing on a box.
If the objective function is identified as convex on the box, any local minimum is also global, and the search within the box becomes easier. Similarly, if the function is convex nowhere on the box and the box lies inside the feasible set, then the box can be removed as it contains no local, and hence also no global, minimum.
Convexity also plays an important role in the global optimization  $\alpha$BB method \cite{Flo2000,FloGou2009,FloPar2009,Hla2013ca,SkjWes2014}, which is based in constructing a convex underestimator of the objective function by appending an additional convex quadratic term.

Convexity of the objective function on a box is usually studied via an interval matrix enclosing all Hessian matrices of the function on the box. Since convexity of a function corresponds to positive (semi-)definiteness of its Hessian matrix, we face the problem of checking  positive (semi-)definiteness of an interval matrix.

Let us introduce some notation now. 
We use $\diag(z)$ for the diagonal matrix with entries $z_1,\dots,z_n$, and $\sgn(r)$ for the sign of $r$ ($\sgn(r)=1$ if $r\geq0$ and $\sgn(r)=-1$ otherwise). For vectors, the sign is applied entrywise.

\emph{An interval matrix} $\imace{A}$ is defined as
$$
\imace{A}:=[\umace{A},\omace{A}]=\{A\in\R^{m\times n}\mmid \umace{A}\leq A\leq\omace{A}\},
$$
where $\umace{A},\omace{A}\in\R^{m\times n}$, $\umace{A}\leq\omace{A}$, are given, and the inequality is understood entrywise. \emph{The midpoint} and \emph{the radius} of $\imace{A}$ are defined respectively as
$$
\Mid{A}:=\frac{1}{2}(\umace{A}+\omace{A}),\quad
\Rad{A}:=\frac{1}{2}(\omace{A}-\umace{A}).
$$
The set of all $m$-by-$n$ interval matrices is denoted by $\IR^{m\times n}$. Supposing that both $\Mid{A}$ and $\Rad{A}$ are symmetric, the symmetric counterpart to $\imace{A}$ is 
$$
\smace{A}:=\{A\in\imace{A}\mmid A=A^T\}.
$$

%\begin{definition}

%\subsubsection*{Interval notation and definitions.}
A symmetric interval matrix $\smace{A}\in\IR^{n\times n}$ is \emph{strongly positive definite (positive semidefinite)} if $A$ is positive definite (positive semidefinite) for each $A\in\smace{A}$.
Next, $\smace{A}$ is \emph{weakly positive definite (positive semidefinite)} if $A$ is positive definite (positive semidefinite) for some $A\in\smace{A}$.
Eventually, $\imace{A}\in\IR^{n\times n}$ is \emph{regular} if every $A\in\imace{A}$ is nonsingular.
%\end{definition}

%\subsubsection*{Known results.}
The classical results characterizing strong positive semidefiniteness and positive definiteness are stated below; see Rohn~\cite{Roh1994b,Roh2012a,Roh2012b}, and  Bia{\l}as and  Garloff~\cite{BiaGar1984}. Suppose that $\imace{A}\in\IR^{n\times n}$  is given with $\Mid{A}$ and $\Rad{A}$ symmetric.

\begin{theorem}\label{thmPSD}
The following are equivalent:
\begin{enumerate}[(1)]
\item
$\smace{A}$ is positive semidefinite,
\item
$\Mid{A}-\diag(z)\Rad{A}\diag(z)$ is positive semidefinite for each $z\in\{\pm1\}^n$,
\item
$x^T\Mid{A}x-|x|^T\Rad{A}|x|\geq0$ for each $x\in\R^n$.
\end{enumerate}
\end{theorem}

\begin{theorem}\label{thmPDint}
The following are equivalent:
\begin{enumerate}[(1)]
\item
$\smace{A}$ is positive definite,
\item
$\Mid{A}-\diag(z)\Rad{A}\diag(z)$ is positive definite for each $z\in\{\pm1\}^n$,
\item
$x^T\Mid{A}x-|x|^T\Rad{A}|x|>0$ for each $0\not=x\in\R^n$,
\item\label{thmPDint4}
$\Mid{A}$ is positive definite and $\imace{A}$ is regular.
\end{enumerate}
\end{theorem}

Checking strong  positive (semi-)definiteness is known to be a co-NP-hard problem (Kreinovich et al.\ \cite{KreLak1998}).
%Contrary to checking positive semidefiniteness, 
On the other hand, checking whether there is a positive semidefinite matrix in $\smace{A}$ is a polynomial time problem; see Jaulin and Henrion \cite{JauHen2005}.

There are other related results on positive definiteness of interval matrices. For instance, Liu \cite{Liu2009} presents a sufficient condition and applies it to stability issues, Kolev \cite{Kol2007} presents a method to determine a positive definite margin of an interval matrix, and Shao and Hou \cite{ShaHou2010} propose a necessary and sufficient criterion for a larger class of complex Hermitian interval matrices.

Positive (semi-)definiteness closely relates to matrix eigenvalues. A real symmetric matrix $A$ is positive (semi-)definite if and only if all its eigenvalues are positive (nonnegative). This relation indicates that positive (semi-)definiteness can be investigated from the perspective of eigenvalues of interval matrices. Such eigenvalues were studied, e.g., in \cite{Hla2013a,HlaDan2010,Kol2006,Kol2010,Len2014,MatPas2012,Mon2011}, and some of those results could possibly be used to check for positive (semi-)definiteness; a simple sufficient condition for strong positive definiteness appeared already in Rohn \cite{Roh1994b,Roh2012a,Roh2012b}.
This paper, however, is focused in other direction. We generalize some of the classical results to interval matrices affected by linear dependencies between the matrix entries.

%%%%%%%%%%%%%%%%%%%%%%%%%%%%%%%%%%%%%%%%%%%%%%%%%%%%%%%%%%%%%%% 
% parametric PSD
%%%%%%%%%%%%%%%%%%%%%%%%%%%%%%%%%%%%%%%%%%%%%%%%%%%%%%%%%%%%%%% 
\section{Linear Parametric Matrices: Positive Semidefiniteness}\label{sPsdPar}

The standard notion of an interval matrix assumes that all matrix entries vary within the given intervals independently  of other entries. This assumption is rarely satisfied in practice. To approach more closely to practical use and to model possible dependencies, consider a more general concept of a linear parametric matrix
\begin{align*}
A(p)=\sum_{k=1}^K A^{(k)}p_k,
\end{align*}
where $A^{(1)},\dots,A^{(K)}\in\R^{n\times n}$ are fixed symmetric matrices and $p_1,\dots,p_K$ are parameters varying respectively in $\inum{p}_1,\dots,\inum{p}_K\in\IR$. 

Strong and weak positive definiteness extends to parametric matrices naturally as follows.

\begin{definition}
A parametric matrix $A(p)$, $p\in\ivr{p}$, is \emph{strongly positive definite (positive semidefinite)} if $A(p)$ is positive definite (positive semidefinite) for each $p\in\ivr{p}$.
%A parametric matrix $A(p)$, $p\in\ivr{p}$,
It is \emph{weakly positive definite (positive semidefinite)} if $A(p)$ is positive definite (positive semidefinite) for at least one $p\in\ivr{p}$.
%A parametric matrix $A(p)$, $p\in\ivr{p}$,
%It is regular is  $A(p)$ is nonsingular for every  $p\in\ivr{p}$.
\end{definition}

Linear parametric form generalizes the standard interval matrix. Evaluation $A(\ivr{p})=\sum_{k=1}^K A^{(k)}\inum{p}_k$ by interval arithmetic encloses the set of matrices $A(p)$, $p\in\ivr{p}$, in an interval matrix and reduces the problem to the standard non-parametric one. This \quo{relaxation} of parametric structure, however, overestimates the true set and may lead to loss of positive (semi-)definiteness.

\begin{example}
Let 
$$
A(p)=\begin{pmatrix}1&1\\1&1\end{pmatrix}p,\quad p\in\inum{p}=[0,1].
$$
This parametric matrix is strongly positive semidefinite, but its relaxation
$$
A(\inum{p})=\begin{pmatrix}[0,1]&[0,1]\\{}[0,1]&[0,1]\end{pmatrix}
$$
is not, as it contains, e.g., the indefinite matrix
$$
\begin{pmatrix}0&1\\1&0\end{pmatrix}.%\in A(\inum{p}).
$$

\end{example}

Linear parametric forms are also used to model linear dependencies between parameters in interval linear equation solving
\cite{Hla2012d,Pop2012,PopHla2013,ZimKra2012}.
%; see Hlad\'{\i}k \cite{Hla2012d}, Popova \cite{Pop2012}, or Zimmer et al.\ \cite{ZimKra2012} for instance.
Linear dependencies cause not only the problem to be more difficult from the computational viewpoint, but it is also hard to describe the corresponding solution set; see Mayer \cite{May2012}. 

%%%
\subsection{Strong positive semidefiniteness}

Surprisingly, characterization of strong positive semidefiniteness from Theorem~\ref{thmPSD} can be extended to parametric matrices quite straightforwardly.

\begin{theorem}\label{thmPsdPar}
The following are equivalent:
\begin{enumerate}[(1)]
\item
%$A(p)$, $p\in\ivr{p}$, is strongly positive semidefinite,
$A(p)$ is positive semidefinite for each $p\in\ivr{p}$,
\item
$A(p)$ is positive semidefinite for each $p$ such that $p_k\in\{\unum{p}_k,\onum{p}_k\}$ $\forall k$,
\item
$x^TA(\Mid{p})x-\sum_{k=1}^K|x^TA^{(k)}x|\cdot\Rad{p}_k\geq0$ for each $x\in\R^n$.
\end{enumerate}
\end{theorem}

\begin{proof}\mbox{}

\quo{$(1)\Rightarrow (2)$}
Obvious.

\quo{$(2)\Rightarrow (3)$}
Let $0\not=x\in\R^n$. Define $s_k:=\sgn(x^TA^{(k)}x)$  and $p_k:=\Mid{p}_k-s_k\Rad{p}_k\in\{\unum{p}_k,\onum{p}_k\}$, $k=1,\dots,K$. Now, by positive semidefiniteness of $A(p)$, we have
\begin{align*}
x^TA(\Mid{p})x-\sum_{k=1}^K|x^TA^{(k)}x|\cdot\Rad{p}_k
&=x^TA(\Mid{p})x-\sum_{k=1}^Kx^TA^{(k)}x s_k\Rad{p}_k\\
&=\sum_{k=1}^Kx^TA^{(k)}x p_k
 =x^TA(p)x\geq0.
\end{align*}

\quo{$(3)\Rightarrow (1)$}
Let $p\in\ivr{p}$ and  $x\in\R^n$. Now,
\begin{align*}
x^TA(p)x
&=\sum_{k=1}^Kx^TA^{(k)}x p_k
=x^TA(\Mid{p})x+\sum_{k=1}^Kx^TA^{(k)}x (p_k-\Mid{p}_k)\\
&\geq x^TA(\Mid{p})x-\sum_{k=1}^K|x^TA^{(k)}x|\cdot|p_k-\Mid{p}_k|\\
&\geq x^TA(\Mid{p})x-\sum_{k=1}^K|x^TA^{(k)}x|\cdot\Rad{p}_k
\geq0.
\qedhere
\end{align*}
\end{proof}

This result shows that strong positive semidefiniteness can be verified by checking positive semidefiniteness of $2^K$ real matrices. This enables us to effectively check  strong positive semidefiniteness of large matrices provided the number of parameters is small. Moreover, as stated below, the number $2^K$ can be further reduced in some cases.

\begin{theorem}\label{thmPsdParSuf}
\mbox{}
\begin{enumerate}[(1)]
\item
If $A^{(i)}$ is positive semidefinite for some $i$, then we can fix $p_i:=\unum{p}_i$ for checking strong positive semidefiniteness.
\item
If $A^{(i)}$ is negative semidefinite for some $i$, then we can fix $p_i:=\onum{p}_i$ for checking strong positive semidefiniteness.
\end{enumerate}
\end{theorem}

\begin{proof}\mbox{}

(1) Let $p\in\ivr{p}$. We use the fact that positive semidefiniteness is closed under addition and nonnegative multiples. Thus, $A^{(i)}(p_i-\unum{p}_i)$ is positive definite.
If
$$
\sum_{k\not=i}A^{(k)}p_k+A^{(i)}\unum{p}_i
$$ 
is positive semidefinite for some $p_k\in\inum{p}_k$, $k\not=i$, then 
$$
\sum_{k\not=i}A^{(k)}p_k+A^{(i)}\unum{p}_i+A^{(i)}(p_i-\unum{p}_i)
=A(p)
$$ 
is positive semidefinite, too.

(2) Analogously.
\end{proof}

As long as $K$ is too large to apply Theorem~\ref{thmPsdPar}, and Theorem~\ref{thmPsdParSuf} fails to reduce the number of real matrices to be processed, the following sufficient condition may be useful.

\begin{theorem}\label{thmPsdSuff}
For each $k=1,\dots,K$, let $A^{(k)}=A^{(k)}_1-A^{(k)}_2$, where both $A^{(k)}_1,A^{(k)}_2$ are positive semidefinite. Then $A(p)$, $p\in\ivr{p}$, is strongly positive semidefinite if 
$$
\sum_{k=1}^K\left(A^{(k)}_1\unum{p}_k-A^{(k)}_2\onum{p}_k\right)
$$
is positive semidefinite.
\end{theorem}

\begin{proof}
Let $p\in\ivr{p}$. By closedness of positive semidefiniteness under addition and nonnegative multiples, we have that
\begin{align*}
A(p)
&=\sum_{k=1}^K\left(A^{(k)}_1{p}_k-A^{(k)}_2{p}_k\right)\\
&=\sum_{k=1}^K\left(A^{(k)}_1\unum{p}_k-A^{(k)}_2\onum{p}_k\right)
+\sum_{k=1}^K\left(A^{(k)}_1(p_k-\unum{p}_k)
  +A^{(k)}_2(\onum{p}_k-p_k)\right)
\end{align*}
is  positive semidefinite, too.
\end{proof}

A splitting of $A^{(k)}$ into a difference between two positive semidefinite matrices can be carried out as follows. Let $A^{(k)}=Q\Lambda Q^T$ be a spectral decomposition of $A^{(k)}$. Let $\Lambda^+$ be the diagonal matrix the entries of which are the positive parts of $\Lambda$, and similarly  $\Lambda^-$ has the negative parts on the diagonal. Then $A^{(k)}=Q\Lambda Q^T=Q\Lambda^+ Q^T-Q\Lambda^- Q^T$ and both $Q\Lambda^+ Q^T$, $Q\Lambda^- Q^T$ are positive semidefinite.

%%%
\subsection{Weak positive semidefiniteness}

Concerning weak positive semidefiniteness, the problem is still solvable in polynomial time by utilizing a suitable semidefinite program \cite{GarMat2012,NesNem1994,VanBoy1996}. Let $M(p)$ be the block diagonal matrix with blocks
\begin{align*}
A(p),\ p_1-\unum{p}_1,\ \dots,\ p_K-\unum{p}_K,\
 \onum{p}_1-p_1,\ \dots,\ \onum{p}_K-p_K.  
\end{align*}
All entries of $M(p)$ depend affinely on variables $p$. Positive definiteness of $M(p)$ is equivalent to positive definiteness of $A(p)$ and feasibility of variables $p\in\ivr{p}$. Therefore, by solving this semidefinite program we check whether $A(p)$, $p\in\ivr{p}$, is weakly positive semidefinite.

Anyway, a cheap necessary condition may be useful, e.g., for nonconvexity testing in global optimization \cite{HanWal2004}.

\begin{theorem}\label{thmWeakPsdNec}
For each $k=1,\dots,K$, let $A^{(k)}=A^{(k)}_1-A^{(k)}_2$, where both $A^{(k)}_1,A^{(k)}_2$ are positive semidefinite. If $A(p)$, $p\in\ivr{p}$, is weakly positive semidefinite, then
$$
\sum_{k=1}^K\left(A^{(k)}_1\onum{p}_k-A^{(k)}_2\unum{p}_k\right)
$$
is positive semidefinite.
\end{theorem}

\begin{proof}
Let $p\in\ivr{p}$ such that $A(p)$ is positive semidefinite. By closedness of positive semidefiniteness under addition and nonnegative multiples, we have that
\begin{align*}
A(p)+\sum_{k=1}^K\left(A^{(k)}_1
  (\onum{p}_k-p_k)+A^{(k)}_2(p_k-\unum{p}_k)\right)
=\sum_{k=1}^K\left(A^{(k)}_1\onum{p}_k-A^{(k)}_2\unum{p}_k\right)
\end{align*}
is  positive semidefinite, too.
\end{proof}

In view of Theorem~\ref{thmPsdParSuf}, it is easy to see that the conditions from Theorems~\ref{thmPsdSuff} and~\ref{thmWeakPsdNec} are necessary and sufficient provided for each $k=1,\dots,K$, the matrix $A^{(k)}$ is either positive or negative definite.

%%%%%%%%%%%%%%%%%%%%%%%%%%%%%%%%%%%%%%%%%%%%%%%%%%%%%%%%%%%%%%% 
% parametric PD
%%%%%%%%%%%%%%%%%%%%%%%%%%%%%%%%%%%%%%%%%%%%%%%%%%%%%%%%%%%%%%% 
\section{Linear Parametric Matrices: Positive Definiteness}

In a similar fashion as in Section~\ref{sPsdPar}, we can characterize positive definiteness of parametric matrices.

\begin{theorem}\label{thmPdPar}
The following are equivalent:
\begin{enumerate}[(1)]
\item
$A(p)$, $p\in\ivr{p}$, is strongly positive definite,
\item
$A(p)$ is positive definite for each $p$ such that $p_k\in\{\unum{p}_k,\onum{p}_k\}$ $\forall k$,
\item
$x^TA(\Mid{p})x-\sum_{k=1}^K|x^TA^{(k)}x|\cdot\Rad{p}_k>0$ for each $0\not=x\in\R^n$.
\end{enumerate}
\end{theorem}

\begin{proof}
Analogous to Theorem~\ref{thmPsdPar}.
%\qed
\end{proof}

\begin{theorem}\label{thmPdParSuf}
\mbox{}
\begin{enumerate}[(1)]
\item
If $A^{(i)}$ is positive semidefinite for some $i$, then we can fix $p_i:=\unum{p}_i$ for checking strong positive definiteness.
\item
If $A^{(i)}$ is negative semidefinite for some $i$, then we can fix $p_i:=\onum{p}_i$ for checking strong positive definiteness.
\end{enumerate}
\end{theorem}

\begin{proof}
Analogous to Theorem~\ref{thmPsdParSuf}.
%\qed
\end{proof}

\begin{theorem}\label{thmPdSuff}
For each $k=1,\dots,K$, let $A^{(k)}=A^{(k)}_1-A^{(k)}_2$, where both $A^{(k)}_1,A^{(k)}_2$ are positive semidefinite. Then $A(p)$, $p\in\ivr{p}$, is strongly positive definite if 
$$
\sum_{k=1}^K\left(A^{(k)}_1\unum{p}_k-A^{(k)}_2\onum{p}_k\right)
$$
is positive definite.
\end{theorem}

\begin{proof}
Analogous to Theorem~\ref{thmPsdSuff}.
%\qed
\end{proof}

A parametric matrix $A(p)$, $p\in\ivr{p}$, is called \emph{regular} if $A(p)$ is nonsingular for each $p\in\ivr{p}$.
Regularity of parametric matrices was investigated by Popova \cite{Pop2004b}, for instance.
We have the following relation to regularity, extending item \nref{thmPDint4} of  Theorem~\ref{thmPDint}.

\begin{theorem}\label{thmPdReg}
The parametric matrix $A(p)$, $p\in\ivr{p}$, is strongly positive definite if and only if the following two properties hold:
\begin{enumerate}[(1)]
\item
$A(p)$ is  positive definite for an arbitrarily chosen $p\in\ivr{p}$,
\item
$A(p)$, $p\in\ivr{p}$, is regular.
\end{enumerate}
\end{theorem}

\begin{proof}\mbox{}

\quo{$\Rightarrow$}
Obvious as each positive definite matrix is nonsingular.

\quo{$\Leftarrow$} 
Let $A(p^1)$ be positive definite for $p^1\in\ivr{p}$ and suppose to the contrary that  $A(p^2)$ is not positive definite for $p^2\in\ivr{p}$. Hence $A(p^1)$ has positive eigenvalues, and $A(p^2)$ has at least one non-positive eigenvalue. Due to continuity of eigenvalues \cite{Mey2000} and compactness of $\ivr{p}$, there is $p^0\in\ivr{p}$ such that $A(p^0)$ is singular. A contradiction.
\end{proof}

Now, we have two sufficient conditions for checking strong positive definiteness. The first one is stated in Theorem~\ref{thmPdSuff}, and the second one utilizes regularity according to Theorem~\ref{thmPdReg}. By Poljak and Rohn \cite{PolRoh1993} (see also \cite{Fie2006,KreLak1998}), checking regularity is a co-NP-hard problem even for standard interval matrices. However, there are some polynomially verifiable sufficient conditions that we can apply. We will utilize a similar approach to \cite{Pop2004b}: preconditioning, relaxation and call of a Beeck criterion for checking regularity of an interval matrix (see Rex and Rohn \cite{RexRoh1998}).
By interval aritmetic, we evaluate the interval matrix
$$
\imace{M}:=\sum_{k=1}^K\left(CA^{(k)}\right)\inum{p}_k,
$$
where $C=A(\Mid{p})^{-1}$ is the preconditioner.
Now, the Beeck sufficient condition for regularity of $\imace{M}$ (which implies regularity of $A(p)$, $p\in\ivr{p}$) reads 
\begin{align}\label{regCond}
\rho(\Rad{M})<1.
\end{align}

We show by examples that no one condition for checking strong positive definiteness is stronger than the other one, where the criterion \nref{regCond} is utilized for regularity checking. Notice that the values in the matrices below are displayed to a precision of four digits, however, the real computation was done rigorously in Matlab using the interval library Intlab by Rump \cite{Rum1999} and the verification software package Versoft by Rohn \cite{versoft}.

\begin{example}
Let 
$$
A(p)=
 \begin{pmatrix}1.5&0\\0&1.1\end{pmatrix}p_1+
 \begin{pmatrix}-1&1\\1&1\end{pmatrix}p_2
,\quad p\in\inum{p}=(1,[0,1]).
$$
This parametric matrix is strongly positive definite.

Let us check the sufficient condition by Theorem~\ref{thmPdSuff}. The matrix $A^{(1)}$ is positive definite, so we split only
$$
A^{(2)}=A^{(2)}_1-A^{(2)}_2
=\begin{pmatrix}0.2071 & 0.5\\0.5 & 1.2071\end{pmatrix}
-\begin{pmatrix}1.2071 & -0.5\\-0.5 & 0.2071\end{pmatrix}
$$
Now,
$$
A^{(1)}\cdot 1+A^{(2)}_1\cdot 0-A^{(2)}_2\cdot 1
=\begin{pmatrix}0.2929 & 0.5\\0.5 & 0.8929\end{pmatrix}
$$
is positive definite, proving strong positive definiteness of  $A(p)$, $p\in\ivr{p}$.

In comparison, the sufficient regularity condition \nref{regCond} fails to prove regularity. Using the preconditioner 
$$
C:=A(\Mid{p})^{-1}
=\begin{pmatrix}
%1.02 & 0.5\\0.5 & 1.62
1 & 0.5\\0.5 & 1.6
\end{pmatrix}^{-1},
$$
the relaxation leads to an interval matrix
$$
\imace{M}=\sum_{k=1}^K\left(CA^{(k)}\right)\inum{p}_k
=\begin{pmatrix}
%  [0.2442, 1.7558] & [-0.3993, 0.3993]\\{}
%  [-0.5419, 0.5419] & [0.8146, 1.1854]
[  0.2222,  1.7778] & [-0.4075, 0.4075]\\{} 
[ -0.5556,  0.5556] & [ 0.8148, 1.1852] 
\end{pmatrix},
$$
which is not confirmed to be regular by using the condition \nref{regCond} as $\rho(\Rad{M})= 1.0419\not<1$. 
\end{example}

\begin{example}
Let 
$$
A(p)=
 \begin{pmatrix}3.3&0.25\\0.25&3.3\end{pmatrix}p_1+
 \begin{pmatrix}1&2\\2&0\end{pmatrix}p_2+
 \begin{pmatrix}0&2\\2&1\end{pmatrix}p_3
,\quad p\in\inum{p}=(1,[0,1],[0,1]).
$$
In this example, Theorem~\ref{thmPdSuff} fails to prove positive definiteness. In contrast, with the preconditioner 
$$
C:=A(\Mid{p})^{-1}
=\begin{pmatrix}
3.8 & 2.25\\2.25 & 3.8
%0.4052 & -0.24\\-0.24 & 0.4052
\end{pmatrix}^{-1},
$$
and the relaxation matrix
$$
\imace{M}=\sum_{k=1}^K\left(CA^{(k)}\right)\inum{p}_k
=\begin{pmatrix} 
[0.7227, 1.2773]  & [ -0.6905, 0.6905]\\{} 
[-0.6905,  0.6905]& [  0.7227, 1.2773] 
\end{pmatrix},
$$
the condition \nref{regCond} proves positive definiteness by showing $\rho(\Rad{M})=0.9678<1$. 
\end{example}

Theorem~\ref{thmWeakPsdNec} is modified to necessary condition for weak positive definiteness as follows.

\begin{theorem}\label{thmWeakPdNec}
For each $k=1,\dots,K$, let $A^{(k)}=A^{(k)}_1-A^{(k)}_2$, where both $A^{(k)}_1,A^{(k)}_2$ are positive semidefinite. If $A(p)$, $p\in\ivr{p}$, is weakly positive definite, then
$$
\sum_{k=1}^K\left(A^{(k)}_1\onum{p}_k-A^{(k)}_2\unum{p}_k\right)
$$
is positive definite.
\end{theorem}

\begin{proof}
Analogous to Theorem~\ref{thmWeakPsdNec}.
\end{proof}

%%%%%%%%%%%%%%%%%%%%%%%%%%%%%%%%%%%%%%%%%%%%%%%%%%%%%%%%%%%%%%% 
% example
%%%%%%%%%%%%%%%%%%%%%%%%%%%%%%%%%%%%%%%%%%%%%%%%%%%%%%%%%%%%%%% 
\section{Example}

Consider a class of functions
\begin{align*}
f(x)=\sum_{\ell=1}^Lc_{\ell}x_{i_{\ell}}x_{j_{\ell}}x_{k_{\ell}},
\end{align*}
where $i_{\ell},j_{\ell},k_{\ell}\in\sseznam{0}{n}$ are not necessarily mutually different, and $x_0=1$. For such functions, their Hessian matrix has directly a linear parametric form without using any kind of linearization. It is easy to see that each entry of the Hessian of $f(x)$ is a linear function with respect to $x\in\R^n$. Thus the variables $x$ play the role of the parameters $p$, and their domain $\ivr{x}$ works as $\ivr{p}$.

\begin{example}
Let
$$
f(x,y,z)=x^3+2x^2y-xyz+3yz^2+5y^2,
$$
and we want to check its convexity on  $x\in\inum{x}=[2,3]$, $y\in\inum{y}=[1,2]$ and $z\in\inum{z}=[0,1]$. The Hessian of $f$ reads
$$
\nabla^2f(x,y,z)=
\begin{pmatrix}
6x + 4y & 4x - z & -y \\
4x - z & 10  & -x + 6 z\\
-y & -x + 6z & 6 y 
\end{pmatrix}
$$
The direct evaluation the Hessian matrix by interval arithmetic results in an enclose by the interval matrix
$$
\nabla^2f(\inum{x},\inum{y},\inum{z})\subseteq
\begin{pmatrix}
[16,26] & [7,12] & -[1,2]\\{}
[7,12]  & 10 & [-3,4]\\{}
-[1,2]  & [-3,4] & [6,12]
\end{pmatrix}
$$
This interval matrix is not strongly positive semidefinite since the smallest eigenvalue, computed by the exponential formula by Hertz \cite{Her1992}, is $-2.8950$. Nevertheless, Theorem~\ref{thmPdSuff} proves $\nabla^2f(x,y,z)$ to be positive definite by utilizing the parametric form
$$
\nabla^2f(x,y,z)=
\begin{pmatrix}6& 4& 0\\ 4& 0& -1\\ 0& -1& 0\end{pmatrix}x+
\begin{pmatrix}4& 0& -1\\ 0& 0& 0\\ -1& 0& 6\end{pmatrix}y+
\begin{pmatrix}0& -1& 0\\ -1& 0& 6\\ 0& 6& 0\end{pmatrix}z+
\begin{pmatrix}0& 0& 0\\0& 10& 0\\0& 0& 0\end{pmatrix}.
$$
Thus, we can conclude that $f$ is convex on the interval domain.
\end{example}

\subsubsection*{Acknowledgments.} 

The author was supported by the Czech Science Foundation Grant P402-13-10660S.

%%%%%%%%%%%%%%%%%%%%%%%%%%%%%%%%%%%%%%%%%%%%%%%%%%%%%%%%%%%%%%% 
% REFERENCES
%%%%%%%%%%%%%%%%%%%%%%%%%%%%%%%%%%%%%%%%%%%%%%%%%%%%%%%%%%%%%%% 

\bibliographystyle{abbrv}
\bibliography{coprod_psdpar_lncs}

\begin{thebibliography}{10}

\bibitem{BiaGar1984}
S.~Bia{\l}as and J.~Garloff.
\newblock Intervals of {P}-matrices and related matrices.
\newblock {\em Linear Algebra Appl.}, 58:33--41, 1984.

\bibitem{Fie2006}
M.~Fiedler, J.~Nedoma, J.~Ram\'{\i}k, J.~Rohn, and K.~Zimmermann.
\newblock {\em Linear Optimization Problems with Inexact Data}.
\newblock Springer, New York, 2006.

\bibitem{Flo2000}
C.~A. Floudas.
\newblock {\em Deterministic Global Optimization. Theory, Methods and
  Applications}, volume~37 of {\em Nonconvex Optimization and its
  Applications}.
\newblock Kluwer, Dordrecht, 2000.

\bibitem{FloGou2009}
C.~A. Floudas and C.~E. Gounaris.
\newblock A review of recent advances in global optimization.
\newblock {\em J. Glob. Optim.}, 45(1):3--38, 2009.

\bibitem{FloPar2009}
C.~A. Floudas and P.~M. Pardalos, editors.
\newblock {\em Encyclopedia of Optimization. 2nd ed.}
\newblock Springer, New York, 2009.

\bibitem{GarMat2012}
B.~G{\"{a}}rtner and J.~Matou{\v{s}}ek.
\newblock {\em Approximation Algorithms and Semidefinite Programming}.
\newblock Springer, Berlin Heidelberg, 2012.

\bibitem{HanWal2004}
E.~R. Hansen and G.~W. Walster.
\newblock {\em Global Optimization Using Interval Analysis}.
\newblock Marcel Dekker, New York, second edition, 2004.

\bibitem{HenTot2010}
E.~M.~T. Hendrix and B.~Gazdag-T{\'o}th.
\newblock {\em Introduction to nonlinear and global optimization}, volume~37 of
  {\em Optimization and Its Applications}.
\newblock Springer, New York, 2010.

\bibitem{Her1992}
D.~Hertz.
\newblock The extreme eigenvalues and stability of real symmetric interval
  matrices.
\newblock {\em IEEE Trans. Autom. Control}, 37(4):532--535, 1992.

\bibitem{Hla2012d}
M.~Hlad{\'\i}k.
\newblock Enclosures for the solution set of parametric interval linear
  systems.
\newblock {\em Int. J. Appl. Math. Comput. Sci.}, 22(3):561--574, 2012.

\bibitem{Hla2013a}
M.~Hlad\'{\i}k.
\newblock Bounds on eigenvalues of real and complex interval matrices.
\newblock {\em Appl. Math. Comput.}, 219(10):5584--5591, 2013.

\bibitem{Hla2013ca}
M.~Hlad\'{\i}k.
\newblock On the efficient {Gerschgorin} inclusion usage in the global
  optimization {$\alpha$BB} method.
\newblock {\em J. Glob. Optim.}, 2014.
\newblock DOI 10.1007/s10898-014-0161-7.

\bibitem{HlaDan2010}
M.~Hlad\'{\i}k, D.~Daney, and E.~Tsigaridas.
\newblock Bounds on real eigenvalues and singular values of interval matrices.
\newblock {\em SIAM J. Matrix Anal. Appl.}, 31(4):2116--2129, 2010.

\bibitem{JauHen2005}
L.~Jaulin and D.~Henrion.
\newblock Contracting optimally an interval matrix without loosing any positive
  semi-definite matrix is a tractable problem.
\newblock {\em Reliab. Comput.}, 11(1):1--17, 2005.

\bibitem{Kea1996}
R.~B. Kearfott.
\newblock {\em Rigorous Global Search: Continuous Problems}.
\newblock Kluwer, Dordrecht, 1996.

\bibitem{Kol2006}
L.~V. Kolev.
\newblock Outer interval solution of the eigenvalue problem under general form
  parametric dependencies.
\newblock {\em Reliab. Comput.}, 12(2):121--140, 2006.

\bibitem{Kol2007}
L.~V. Kolev.
\newblock Determining the positive definiteness margin of interval matrices.
\newblock {\em Reliab. Comput.}, 13(6):445--466, 2007.

\bibitem{Kol2010}
L.~V. Kolev.
\newblock Eigenvalue range determination for interval and parametric matrices.
\newblock {\em Int. J. Circuit Theory Appl.}, 38(10):1027--1061, 2010.

\bibitem{KreLak1998}
V.~Kreinovich, A.~Lakeyev, J.~Rohn, and P.~Kahl.
\newblock {\em Computational Complexity and Feasibility of Data Processing and
  Interval Computations}.
\newblock Kluwer, 1998.

\bibitem{Len2014}
H.~Leng.
\newblock Real eigenvalue bounds of standard and generalized real interval
  eigenvalue problems.
\newblock {\em Appl. Math. Comput.}, 232:164--171, 2014.

\bibitem{Liu2009}
W.~Liu.
\newblock Necessary and sufficient conditions for the positive definiteness and
  stability of symmetric interval matrices.
\newblock In {\em Proceedings of the 21st Annual International Conference on
  Chinese Control and Decision Conference}, CCDC'09, pages 4574--4579,
  Piscataway, NJ, USA, 2009. IEEE Press.

\bibitem{MatPas2012}
M.-H. Matcovschi, O.~Pastravanu, and M.~Voicu.
\newblock Right bounds for eigenvalue ranges of interval matrices - estimation
  principles vs global optimization.
\newblock {\em Control Eng. Appl. Inform.}, 14(1):3--13, 2012.

\bibitem{May2012}
G.~Mayer.
\newblock An {Oettli--Prager}-like theorem for the symmetric solution set and
  for related solution sets.
\newblock {\em SIAM J. Matrix Anal. Appl.}, 33(3):979--999, 2012.

\bibitem{Mey2000}
C.~D. Meyer.
\newblock {\em {Matrix Analysis and Applied Linear Algebra}}.
\newblock {SIAM}, {Philadelphia}, 2000.

\bibitem{Mon2011}
M.~M{\"o}nnigmann.
\newblock Fast calculation of spectral bounds for hessian matrices on
  hyperrectangles.
\newblock {\em SIAM J. Matrix Anal. Appl.}, 32(4):1351--1366, 2011.

\bibitem{NesNem1994}
Y.~Nesterov and A.~Nemirovskii.
\newblock {\em Interior-Point Polynomial Algorithms in Convex Programming}.
\newblock SIAM, Philadelphia, 1994.

\bibitem{Neu2004}
A.~Neumaier.
\newblock {Complete search in continuous global optimization and constraint
  satisfaction}.
\newblock {\em Acta Numer.}, 13:271--369, 2004.

\bibitem{PolRoh1993}
S.~Poljak and J.~Rohn.
\newblock Checking robust nonsingularity is {NP}-hard.
\newblock {\em Math. Control Signals Syst.}, 6(1):1--9, 1993.

\bibitem{Pop2004b}
E.~D. Popova.
\newblock Strong regularity of parametric interval matrices.
\newblock In I.~Dimovski~et al., editor, {\em Mathematics and Education in
  Mathematics, Proceedings of the 33rd Spring Conference of the Union of
  Bulgarian Mathematicians, Borovets, Bulgaria}, pages 446--451. BAS, 2004.

\bibitem{Pop2012}
E.~D. Popova.
\newblock Explicit description of {AE} solution sets for parametric linear
  systems.
\newblock {\em SIAM J. Matrix Anal. Appl.}, 33(4):1172--1189, 2012.

\bibitem{PopHla2013}
E.~D. Popova and M.~Hlad{\'\i}k.
\newblock Outer enclosures to the parametric {AE} solution set.
\newblock {\em Soft Comput.}, 17(8):1403--1414, 2013.

\bibitem{RexRoh1998}
G.~Rex and J.~Rohn.
\newblock Sufficient conditions for regularity and singularity of interval
  matrices.
\newblock {\em SIAM J. Matrix Anal. Appl.}, 20(2):437--445, 1998.

\bibitem{Roh1994b}
J.~Rohn.
\newblock Positive definiteness and stability of interval matrices.
\newblock {\em SIAM J. Matrix Anal. Appl.}, 15(1):175--184, 1994.

\bibitem{versoft}
J.~Rohn.
\newblock {VERSOFT: Verification software in MATLAB / INTLAB, version 10},
  2009.

\bibitem{Roh2012a}
J.~Rohn.
\newblock A handbook of results on interval linear problems.
\newblock Technical Report 1163, Institute of Computer Science, Academy of
  Sciences of the Cz ech Republic, Prague, 2012.

\bibitem{Roh2012b}
J.~Rohn.
\newblock A manual of results on interval linear problems.
\newblock Technical Report 1164, Institute of Computer Science, Academy of
  Sciences of the Czech Republic, Prague, 2012.

\bibitem{Rum1999}
S.~M. Rump.
\newblock {INTLAB -- INTerval LABoratory}.
\newblock In T.~Csendes, editor, {\em {Developments~in~Reliable Computing}},
  pages 77--104. Kluwer Academic Publishers, Dordrecht, 1999.

\bibitem{ShaHou2010}
J.~Shao and X.~Hou.
\newblock Positive definiteness of {Hermitian} interval matrices.
\newblock {\em Linear Algebra Appl.}, 432(4):970--979, 2010.

\bibitem{SkjWes2014}
A.~Skj{\"{a}}l and T.~Westerlund.
\newblock New methods for calculating {$\alpha BB$}-type underestimators.
\newblock {\em J. Glob. Optim.}, 58(3):411--427, 2014.

\bibitem{VanBoy1996}
L.~Vandenberghe and S.~Boyd.
\newblock Semidefinite programming.
\newblock {\em SIAM Rev.}, 38(1):49--95, 1996.

\bibitem{ZimKra2012}
M.~Zimmer, W.~Kr{\"a}mer, and E.~D. Popova.
\newblock Solvers for the verified solution of parametric linear systems.
\newblock {\em Comput.}, 94(2-4):109--123, 2012.

\end{thebibliography}

\end{document}